\documentclass[12pt]{amsart}
\usepackage{graphicx} 
\usepackage{amsmath}
\usepackage{amssymb}
\usepackage{mathrsfs}
\usepackage{amscd}
\usepackage{bbm}
\usepackage{youngtab}
\usepackage{xcolor}
\usepackage{enumerate}
\usepackage{multirow}
\usepackage{hyperref}

\newcommand{\A}{\mathcal A}

\newcommand{\Z}{\mathbb Z}
\newcommand{\calZ}{\mathcal Z}
\newcommand{\R}{\mathbb R}
\newcommand{\calB}{\mathcal B}

\newcommand{\be}{\begin{equation}}
\newcommand{\ee}{\end{equation}}
\newcommand{\ds}{\displaystyle}

\newcommand{\citep}[1]{\cite{#1}}

\numberwithin{equation}{section}

\newtheorem{thm}{Theorem}[section]
\newtheorem{cor}[thm]{Corollary}
\newtheorem{lem}[thm]{Lemma}
\newtheorem{prop}[thm]{Proposition}
\newtheorem{defn}[thm]{Definition}
\newtheorem{rem}[thm]{Remark}

\newtheorem{ex}[thm]{Example}

\title{The Hilbert-Galton board}

\author{Arvind Ayyer}
\address{Arvind Ayyer, 
Department of Mathematics, 
Indian Institute of Science,
Bangalore - 560012, India}
\email{arvind@math.iisc.ernet.in}
\author{Sanjay Ramassamy}
\address{Sanjay Ramassamy,
Unit\'e de Math\'ematiques Pures et Appliqu\'ees,
\'Ecole normale sup\'erieure de Lyon,
46 all\'ee d'Italie,
69364 Lyon Cedex 07, France}
\email{sanjay.ramassamy@ens-lyon.fr}

\date{\today}  

\begin{document} 

\begin{abstract}
We introduce the Hilbert-Galton board as a variant of the classical Galton board. Balls fall into a row of bins at a rate depending on the bin, and at random times, each bin gets shifted one unit to the right and an empty bin is added to the left. We compute the stationary distribution of this Markov chain and show the existence of an enriched Markov chain on triangular arrays which projects down to the Hilbert-Galton board. We also define finite-ball projections of the Hilbert-Galton board, for which we compute the stationary distribution, the full spectrum and the grand coupling time.
\end{abstract}

\maketitle

\section{Introduction}
\label{sec:intro}
The Galton board is a mechanical device introduced by Galton~\cite{galton-1894} in order to illustrate the central limit theorem.
Balls fall through an array of regularly placed obstacles into a row of bins. At each obstacle, a ball may go left or right with the same probability, so that the abscissa of the bin into which the ball falls has a binomial distribution, which approximates the normal distribution when the number of obstacles and bins is large. See Figure~\ref{fig:galtonboard} for a small example.

\begin{figure}[htbp!]
\begin{center}
\includegraphics[scale=0.5]{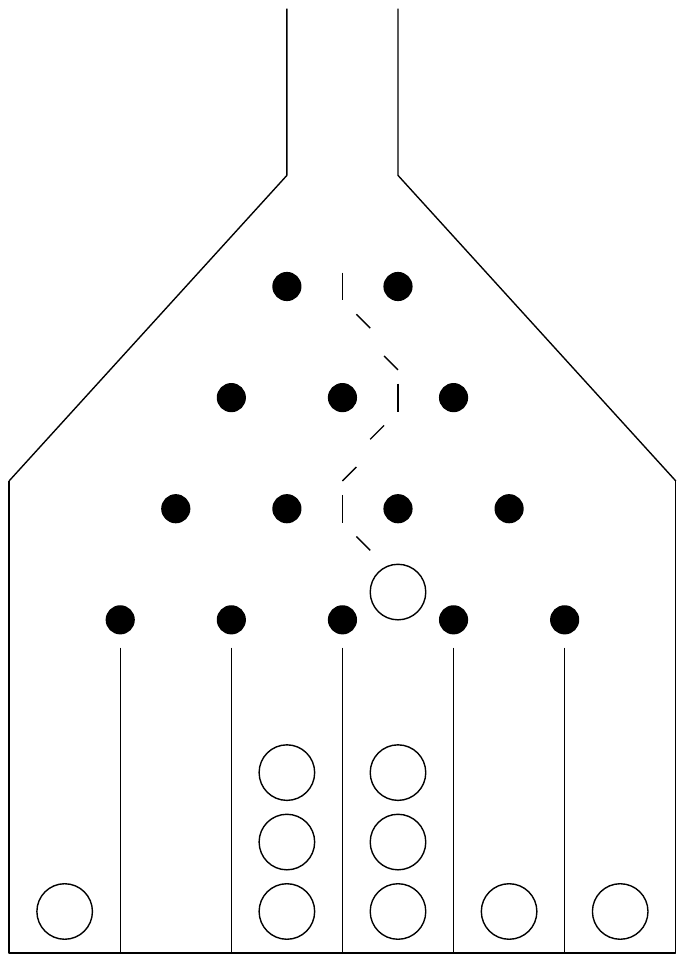}
\caption{An illustration of the Galton board, with obstacles marked by black dots.}
\label{fig:galtonboard}
\end{center}
\end{figure}

Forgetting about the mechanical obstacles, one can simply consider the Galton board as a Markov chain on configurations of balls inside bins, where the rate at which a ball falls into a given bin is a function of the position of the bin. 
We generalise this by considering arbitrary rates of falling inside each bin, and in addition a shift operation as follows.
Consider a row of $n$ bins, labelled from $1$ to $n$ from left to right. Balls fall into bin number $i$ at rate $x_i$ for any $1\leq i \leq n$. At rate $x_0$, a shift to the right occurs, which means that the following events take place simultaneously:
\begin{enumerate}
 \item the rightmost bin gets deleted ;
 \item every bin of index $1 \leq i \leq n-1$ gets shifted to the right by one unit, thus becomes indexed by $i+1$ ;
 \item an empty new bin is added at position $1$.
\end{enumerate}
When a bin gets shifted, the number of balls it contains remains unchanged. 
See Figure~\ref{fig:example-shift} for an example of this shift operation.
We denote this Markov chain by $Z^{(n)}$ and call it the \emph{Hilbert-Galton board}, as a reference to the parable of the Hilbert hotel (see~\cite{gamow-2012}).
This continuous-time Markov chain is irreducible and we compute its stationary distribution (Theorem~\ref{thm:constantbinprob}). We show the existence of a natural Markov chain $X^{(n)}$ on triangular arrays 
which projects down to the Hilbert-Galton board and for which we can compute the stationary distribution (Theorem~\ref{thm:triangle-chain-statioprob}). We call $X^{(n)}$ the \emph{enriched Hilbert-Galton board}.

\begin{figure}[htbp!]
\begin{center}
\includegraphics[scale=0.7]{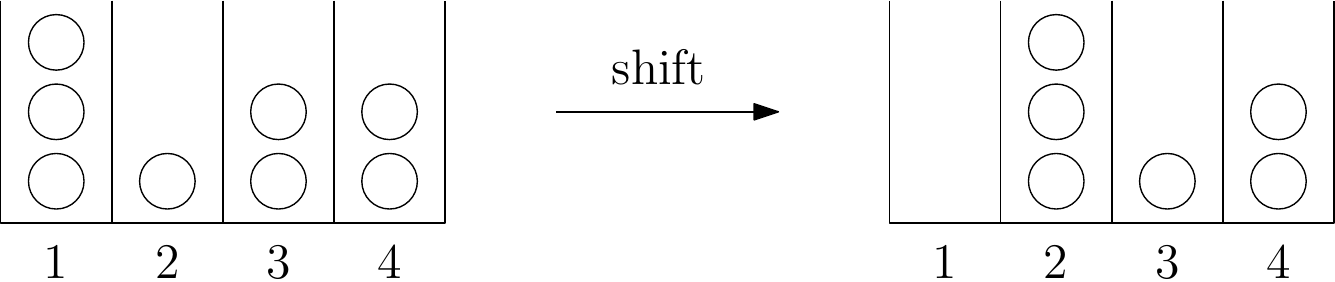}
\caption{An example of the shift transition on $Z^{(4)}$.}
\label{fig:example-shift}
\end{center}
\end{figure}

We also introduce a finite-ball projection of (a variant of) the Hilbert-Galton board. Let $\tilde{Z}^{(n)}$ be the Markov chain defined as the Hilbert-Galton board, with the only difference being that whenever a new bin is added during a shift operation, it arrives containing a single ball (rather than arriving empty). The Markov chain $Y^{(n)}$ is defined by looking at the evolution of the leftmost $n$ balls in $\tilde{Z}^{(n)}$. 
We call $Y^{(n)}$ the \emph{$n$-ball projection of the Hilbert-Galton board}.
One can reformulate it as follows. We consider configurations of $n$ balls placed inside a row of bins such that no bin is empty. For any $i\geq1$, at rate $x_{i-1}$ we add one ball to the left of the $i$'th bin 
and we delete the rightmost ball, so that the total number of balls remains constant. If $i=1$, we need to create a new leftmost bin where we add the new ball. If the rightmost bin becomes empty after deletion of the rightmost ball, we delete the bin.

With this formulation, the $n$-ball projection of the Hilbert-Galton board can be seen as a modification of the infinite-bin model introduced by Foss and Konstantopoulos~\cite{foss-konstantopoulos-2003}, where the new balls were added to the left of the $i$'th ball, rather than to the left of the $i$'th bin\footnote{To be precise, this is actually the finite-dimensional projection of the infinite-bin model, which may  involve configurations with infinitely many balls.}. A special instance of the infinite-bin model had already appeared in work of Aldous and Pitman~\cite{aldous1983asymptotic}. The infinite-bin model was initially defined to study a model of random directed acyclic graphs called the Barak-Erd\H{o}s model~\cite{barak-erdos-1984}, 
 but in recent years it has received quite a bit of attention as an interacting particle system. Properties of interest have been the speed of the front (rate of creation of new leftmost bins)~\cite{foss-konstantopoulos-2003,mallein-ramassamy-2016,aldous1983asymptotic} and the existence of a grand coupling~\cite{chernysh-ramassamy-2017}, which makes it a renewal process and allows perfect sampling of the stationary measure using coupling from the past~\cite{foss-konstantopoulos-2003,foss-zachary-2013}. The stationary probabilities for the infinite-bin model are rational functions of the rates $x_i$, but there is currently no general formula for them. The spectrum of the Markov chain is not known either.

By contrast, the $n$-ball projection of the Hilbert-Galton board is much more tractable. Although it is irreversible, 
 we derive formulas for the stationary probabilities (Theorem~\ref{thm:statioprob}) as well as for the spectrum of the chain (Theorem~\ref{thm:spectrum}). Perfect sampling using coupling from the past is also possible and we derive a formula for the grand coupling time (Proposition~\ref{prop:couplingtime}).
We note in passing that all the Markov chains studied in this work are irreversible. Very few families of irreversible Markov chains are known for which stationary distributions and spectra can be explicitly computed. One such class is that of $\mathscr{R}$-trivial Markov chains~ \cite{ayyer_schilling_steinberg_thiery.2013}. For these, the representation theory of monoids can be used to obtain the stationary distribution, the spectrum and the mixing time. This has been used effectively for sampling linear extension for posets using promotion~\cite{ayyer_klee_schilling.2012} and for understanding nonabelian sandpile models on trees~\citep{ayyer_schilling_steinberg_thiery.sandpile.2013}.

\subsection*{Organization of the paper}
In Section~\ref{sec:summary} we introduce the notation for triangular arrays and use it to state and discuss our main results. We compute in Section~\ref{sec:binmodel} the stationary distribution for the Hilbert-Galton board using two methods: first via a bi-infinite stationary version of the Markov chain, and then using a projection from the enriched Hilbert-Galton board. Finally, in Section~\ref{sec:finiteballprojection}, we focus on the $n$-ball projection of the Hilbert-Galton board. We compute its stationary distribution, its spectrum and give a formula for its grand coupling time.

\section{Statement and discussion of the results}
\label{sec:summary}

\subsection{The Hilbert-Galton board and its enrichment}

Fix an integer $n\geq1$. The state space of the Hilbert-Galton board $Z^{(n)}$ is the set $\Z_{\geq0}^n$, i.e. the set all the $n$-tuples of nonnegative integers $(c_1,\ldots,c_n)$. The integer $c_i$ is interpreted as the number of balls in the $i$'th bin. For any $0\leq j \leq n$, we define the map $U_j:\Z_{\geq0}^n\rightarrow \Z_{\geq0}^n$ as follows: for any $c=(c_1,\ldots,c_n)\in\Z_{\geq0}^n$, set
\[
U_j(c)=
\begin{cases}
(0,c_1,\dots,c_{n-1}) & \text{if $j=0$}, \\
(c_1,\dots,c_{j-1},c_j+1,c_{j+1},\dots,c_n) & \text{if $1 \leq j \leq n$}.
\end{cases}
\]
The Hilbert-Galton board $Z^{(n)}$ is defined to be the continuous-time Markov chain on $\Z_{\geq0}^n
$, where an $n$-tuple $c\in\Z_{\geq0}^n$ makes a transition to $U_j(c)$ at rate $x_j$ for any $0\leq j \leq n$. Unless otherwise stated, we will assume in the remainder of this article that $x_0>0$ and $x_1>0$. The general formula for the stationary probability of any state of $Z^{(n)}$ uses the notion of triangular arrays, which we now define.

\begin{defn}
\label{def:triangles}
For any $n\geq0$, a \emph{triangular array} of size $n$ is a collection of nonnegative integers $A=(A_{k,j})_{1 \leq j \leq k \leq n}$. We denote by $\A_{n}$ the set of all triangular arrays of size $n$. By convention, there exists a single triangular array of size $0$, denoted by $\epsilon$.
\end{defn}

We represent a triangular array by aligning each row to the left. For example, a triangular array $A$ of size $3$ is represented as 
\[
\begin{array}{cccc}
 A_{1,1} \\
 A_{2,1} & A_{2,2} \\
 A_{3,1} & A_{3,2} & A_{3,3}
\end{array}
\]
There are three natural ways to take partial sums of elements in a triangular array: horizontally, vertically or diagonally. For any integers $1 \leq i \leq n$ and for any array $A\in\A_{n}$, we define these sums,
\begin{equation}
\label{arraysums}
\begin{split}
h_i(A) &:= \sum_{j=1}^i A_{i,j}, \\
v_i(A) &:= \sum_{k=i}^{n} A_{k,i}, \\
d_i(A) &:= \sum_{j=1}^{i} A_{n -i +j,j}.
\end{split}
\end{equation}

See Figure~\ref{fig:example-triangle} for an illustration of the various sums in an array.

\begin{figure}[htbp!]
\begin{center}
\includegraphics[scale=0.5]{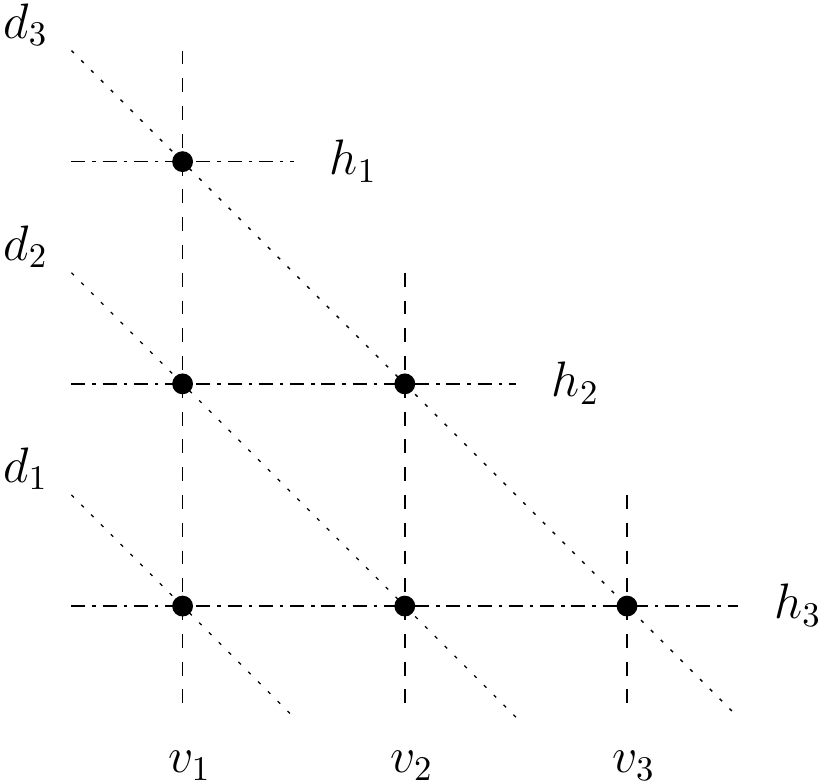}
\caption{An illustration of horizontal, vertical and diagonal sums in a triangular array of size 3.}
\label{fig:example-triangle}
\end{center}
\end{figure}

For any $n$-tuple of nonnegative integers $(c_1,\ldots,c_{n})$, we define
\[
H_{n}(c_1,\ldots,c_{n}):=\left \{A\in\A_n | \; h_i(A)=c_i \; \forall i \in \left[ 1,n \right] \right\}
\]
to be the set of triangular arrays of size $n$ with prescribed row sums. 
Recall that for any nonnegative integers $m_1,\ldots,m_k$, if we write $m=m_1+\cdots+m_k$, then the multinomial coefficient associated with the $k$-tuple $(m_1,\ldots,m_k)$ is defined by
\[
 \binom{m}{m_1, \dots,m_k}:=\frac{m!}{\prod_{i=1}^k m_i!} 
\]
Given a vector $x=(x_0,\ldots,x_{n})$ of nonnegative reals such that $x_0>0$, we define the weight function $w_{x}:\A_{n}\rightarrow \mathbb{R}_{\geq0}$ as follows. For any $A\in\A_{n}$,
\begin{equation}
\label{eq:wx}
 w_{x}(A):=x_0^{n}\prod\limits_{k=1}^{n} \binom{d_k(A)}{A_{n-k+1,1}, \dots, A_{n,k}} 
 \frac{x_k^{v_k(A)}}{y_k^{d_k(A)+1}},
 \end{equation}
 where $y_k:=x_0+x_1+\cdots+x_k$.
By an abuse of notation, if the vector $x$ has more than $n+1$ elements, we can still define the function $w_{x}$ on $\A_{n}$ by restricting the vector $x$ to its first $n+1$ elements. By convention, when $n=0$, we have $w_{x}(\epsilon)=1$. Using this notation, we can now express the stationary probabilities for the Hilbert-Galton board.

\begin{thm}
\label{thm:constantbinprob}
Fix $n\geq1$. Then the Hilbert-Galton board $Z^{(n)}$ is an irreducible positive recurrent Markov chain and the stationary probability of any state $(c_1,\ldots,c_n)\in\Z_{\geq0}^n$ is
\begin{equation}
\label{eq:completeconstantbinprob}
\pi_{Z^{(n)}}(c_1,\ldots,c_n)=\sum_{A\in H_n(c_1,\ldots,c_n)} w_{x}(A).
\end{equation}
More generally, for any $1\leq \ell \leq n$, the stationary probability that the $i$'th bin contains exactly $c_i$ balls for $1 \leq i \leq \ell$  is given by
\begin{equation}
\label{eq:constantbinprob}
\pi_{Z^{(n)}}(c_1,\ldots,c_\ell)=\sum_{A\in H_{\ell}(c_1,\ldots,c_{\ell})} w_{x}(A).
\end{equation}
\end{thm}
Since the stationary probabilities for $Z^{(n)}$ are expressed as sums over triangular arrays, it is natural to ask for the existence of an enriched Markov chain on triangular arrays which would project down to the Hilbert-Galton board. It turns out that such a Markov chain on $\A_n$ indeed exists.

We call this the enriched Hilbert-Galton board and denote it by $X^{(n)}$.
It acts on $\A_n$ as follows. 
For any $A=(A_{k,j})_{1 \leq j \leq k \leq n} \in \A_n$, there are $n+1$ transitions out of $A$:
\[
A \longrightarrow
\begin{cases}
\begin{array}{ccccc}
 0 \\
 A_{1,1} & 0 \\
 A_{2,1} & A_{2,2} & 0 \\
 \vdots & \vdots & \ddots & \ddots \\
 A_{n-1,1} & A_{n-1,2} & \dots & A_{n-1,n-1} &  0,
\end{array}
& \text{with rate $x_0$}, \\
\\
\begin{array}{ccccc}
 A_{1,1} \\
 A_{2,1} & A_{2,2} \\
  \vdots &  & \ddots \\
 A_{i,1} & \dots & \dots & A_{i,i}+1 \\
 \vdots &  &  & \hspace*{0.5cm}\ddots \\
 A_{n,1} &  \dots & \dots & \dots &  A_{n,n}.
\end{array}
& \text{with rate $x_i$ for $1 \leq i \leq n$}.
\end{cases}
\]
We have the following result for the chain $X^{(n)}$.
\begin{thm}
\label{thm:triangle-chain-statioprob}
Fix $n\geq1$ and assume that $x_i>0$ for all $0\leq i \leq n$. Then the Markov chain $X^{(n)}$ is irreducible and positive recurrent. Moreover, for any $A \in \A_n$ the stationary probability of $A$ is given by
\begin{equation}
\label{triangle-statprob}
\pi_{X^{(n)}}(A)= w_x(A).
\end{equation}
Furthermore, defining the map
\[
p_n:A\in\A_n\rightarrow (h_1(A),\ldots,h_n(A))\in\Z_{\geq0}^n,
\]
the projection $p_n(X^{(n)})$ is a Markov chain with the same law as $Z^{(n)}$.
\end{thm}

Theorem~\ref{thm:constantbinprob} follows immediately as a corollary of Theorem~\ref{thm:triangle-chain-statioprob}. The proofs of these results are given in Section~\ref{sec:binmodel}.

\subsection{The finite-ball projection}

We now move on to the $n$-ball projection of the Hilbert-Galton board. 
For any $n\geq1$, the state space of $Y^{(n)}$ is the set $C_n$ of all \emph{compositions} of the integer $n$, i.e the set of all $\ell$-tuples $(c_1,\ldots,c_\ell)$ of positive integers, with $1\leq \ell \leq n$. The integer $\ell$ is called the \emph{length} of the composition $(c_1,\ldots,c_\ell)$. It is well-known that $\#C_n = 2^{n-1}$. 
For any integer $0 \leq j \leq n-1$, define the map $T_j:C_n \rightarrow C_n$ as follows: for any composition $c = (c_1,\dots,c_{\ell}) \in C_n$, we set
\[
T_j(c)=
\begin{cases}
(1,c_1,\dots,c_{\ell-1},c_{\ell}-1) & \text{if $j=0$}, \\
(c_1,\dots,c_{j-1},c_j+1,c_{j+1},\dots,c_{\ell-1},c_{\ell}-1) & \text{if $1 \leq j \leq \ell -1$}, \\
c & \text{if $j \geq \ell$},
\end{cases}
\]
where the last part should be deleted in the first two cases if $c_\ell = 1$.
The chain $Y^{(n)}(t)$ is defined as the continuous-time Markov chain on $C_n$, where a composition $c\in C_n$ makes a transition to $T_j(c)$ at rate $x_j$ for any $0\leq j \leq n-1$. From the stationary distribution of $Z^{(n)}$ one can deduce the stationary distribution of $Y^{(n)}$.

\begin{thm}
\label{thm:statioprob}
The $n$-ball projection of the Hilbert-Galton board $Y^{(n)}$ is irreducible. Furthermore, fix $c=(c_1,\ldots,c_{\ell})\in C_n$ and set $\gamma_i=c_i-1$ for any $1 \leq i \leq \ell$. Then we have the following two equivalent expressions for the stationary probability of $c$ for $Y^{(n)}$:
\begin{align}
\pi_{Y^{(n)}}(c) &= \sum_{A\in H_{\ell-1}(\gamma_1,\ldots,\gamma_{\ell-1})} w_{x}(A) - \sum_{s=0}^{\gamma_{\ell}-1}\sum_{A\in H_{\ell}(\gamma_1,\ldots,\gamma_{\ell-1},s)} w_{x}(A) \label{eq:statioprob1} \\
 &=  \sum_{s=\gamma_{\ell}}^{\infty}\sum_{A\in H_{\ell}(\gamma_1,\ldots,\gamma_{\ell-1},s)} w_{x}(A) \label{eq:statioprob2}
\end{align}
\end{thm}

In formula~\eqref{eq:statioprob1}, there are a finite number of terms, but some are present with a minus sign. In formula~\eqref{eq:statioprob2} all terms have positive coefficients, but there are an infinite number of terms. In this case too, it is possible to construct a Markov chain on triangular arrays which projects down to $Y^{(n)}$. 
The enriched chain is similar to $X^{(n)}$ with a result similar to that of Theorem~\ref{thm:triangle-chain-statioprob}, but the case analysis is a little more complicated. 

Let us illustrate the computation of the stationary distribution of $Y^{(n)}$ by an example.

\begin{ex}
\label{ex:statioprob}
Consider the composition $c=(2,3)$. In this case, $\ell=2$, $\gamma_1=1$ and $\gamma_2=2$. Let us compute $\pi_{Y^{(5)}}(c)$ using formula~\eqref{eq:statioprob1}. The first sum is over a single element, the array of size $1$ made up of the number $1$. The second sum runs over the following three arrays of size $2$:
\[
\begin{array}{cccc}
 1 \\
 0 & 0
\end{array},
\qquad
\begin{array}{cccc}
 1 \\
 1 & 0
\end{array},
\qquad
\begin{array}{cccc}
 1 \\
 0 & 1
\end{array}.
\]
Hence
\begin{align*}
\pi_{Y^{(n)}}(2,3) = & \frac{x_0x_1}{y_1^2}-\frac{x_0^2x_1}{y_1y_2^2}-\frac{x_0^2x_1^2}{y_1^2y_2^2}-2\frac{x_0^2x_1x_2}{y_1y_2^3} \\
= & \frac{x_0 x_1 \left(x_1^3+x_0 x_1^2+3 x_1^2 x_2 +3 x_1 x_2^2 +2 x_0 x_1 x_2 +x_2^3+3 x_0 x_2^2\right)}{y_1^2 y_2^3}.
\end{align*}
\end{ex}

One can deduce from the stationary distributions of the chains $Z^{(n)}$ and $Y^{(n)}$ a series of properties of the $n$-ball projection of the Hilbert-Galton board. Since $Y^{(n)}$ is a finite state chain, the stationary probabilities are multivariate rational functions of
the indeterminates $x_1,\dots,x_n$.
We define the \emph{partition function} $\calZ(Y^{(n)})$ to be the least common multiple of the denominators in the stationary probabilities. Then we have the following results.

\begin{cor}
\label{cor:partitionfunction}
The partition function of the $n$-ball projection of the Hilbert-Galton board is given by
\[
\calZ(Y^{(n)}) = \prod_{k=1}^{n-1} y_k^{n-k}.
\]
\end{cor}

One can also compute the stationary distribution for $Y^{(n)}$ of the number of balls in the leftmost bin.

\begin{cor}
\label{cor:leftmostbin}
For the $n$-ball projection of the Hilbert-Galton board $Y^{(n)}$, the probability that the first bin contains $j$ balls at stationarity is given by
\[
\begin{cases}
\ds \frac{x_0 x_1^{j-1}}{y_1^j} & \text{if $j<n$}, \\[0.4cm]
\ds \left(\frac{x_1}{y_1}\right)^{n-1} & \text{if $j=n$}.
\end{cases}
\]
\end{cor}

As a last application of Theorem~\ref{thm:statioprob}, we identify some instances when the stationary probability is given by a single monomial. For any $\alpha\geq1$ and $\ell \geq 2$, denote by $c_{\ell}(\alpha)$ the composition of length $\ell$, whose first part equals $\alpha$ and whose other parts are all equal to $1$. For example, $c_3(5)=(5,1,1)$. Then the following holds.

\begin{cor}
\label{cor:monomialproba}
\mbox{}
\begin{enumerate}
 \item For any $\alpha\geq1$ and $\ell \geq 2$, we have
 \[
 \pi_{Y^{(n)}}(c_{\ell}(\alpha))=\frac{x_0^{\ell-1}x_1^{\alpha-1}}{y_{\ell-1}^{\alpha}\prod_{k=1}^{\ell-2}y_k}.
 \]
 \item When $x_2 = \cdots = x_{n-1} = 0$, the stationary probability of a composition $(c_1,\ldots,c_\ell)$ only depends on $n$ and $\ell$ and is given by
\[
\frac{x_0^{\ell-1} x_1^{n-\ell}}{y_1^{n-1}}.
\]
\end{enumerate}
\end{cor}

We will prove Corollary~\ref{cor:partitionfunction} in Section~\ref{sec:finiteballprojection}. 
Corollary~\ref{cor:leftmostbin} easily follows from the coupling of $Y^{(n)}$ with $Z^{(n)}$. Corollary~\ref{cor:monomialproba} follows from a careful but elementary analysis of formula~\eqref{eq:statioprob1}. 
We omit the proofs of the last two corollaries.

Given a sequence of transitions, the \emph{grand coupling time} associated with this sequence is defined to be the first time that the $2^{n-1}$ copies of the Markov chain $Y^{(n)}$ starting from all the possible states and evolving according to that sequence of transitions collapse to a single state.
Starting the $n$-ball projection of the Hilbert-Galton board from any state and applying transition $T_0$ $n$ times produces the state $(1,\ldots,1)$. Because of the existence of these coupling events when applying the transitions $T_j$, it is possible to sample exactly from the stationary distribution $\pi_{Y^{(n)}}$ using Propp and Wilson's coupling from the past technique~\cite{propp-wilson-1996}. Given a sequence $(u_i)_{i\geq1}$ of nonnegative integers, one can naturally associate the sequence of transitions $(T_{u_i})_{i\geq1}$ on $Y^{(n)}$.

\begin{ex}
\label{ex:grand coupling}
Take $n=3$ and the sequence $(u_i)_{i\geq1}$ to be $1,2,0,1,\ldots$. Then here is how the four possible Markov chains starting from the four possible initial states behave when they use $(u_i)_{i\geq1}$ as a source of randomness.
\[
\begin{matrix}
& T_1 && T_2 && T_0 && T_1 & \\
(3) &\rightarrow &(3) &\rightarrow &(3) &\rightarrow &(1,2) &\rightarrow &(2,1) \\
(2,1) &\rightarrow &(3) &\rightarrow &(3) &\rightarrow &(1,2) &\rightarrow &(2,1) \\
(1,2) &\rightarrow &(2,1) &\rightarrow &(2,1) &\rightarrow &(1,2) &\rightarrow &(2,1) \\
(1,1,1) &\rightarrow &(2,1) &\rightarrow &(2,1) &\rightarrow &(1,2) &\rightarrow &(2,1)
\end{matrix}
\]
Hence, the grand coupling time associated with the sequence $(u_i)$ is $\tau = 3$.
\end{ex}

It is not hard to see that the coupling time is linear in $n$ in expectation, and hence logarithmic in the number of states of the chain. Indeed, the coupling time is at least (resp. at most) equal to the time needed to have $n-1$ transitions of any type (resp. of type $0$ or $1$) in the chain. We go further and provide a way to compute exactly the grand coupling time associated with a given sequence of transitions. For any $i\geq1$, we say that an element $u_i$ of the sequence is {\em ineffective} if $u_i\geq 2$ and there are at most $u_i-2$ zeros to the left of $u_i$. All the other elements of the sequence are called {\em effective}. Then the following holds.

\begin{prop}
\label{prop:couplingtime}
The grand coupling time $\tau$ associated to a sequence $(u_i)_{i\geq1}$ is the position of the $(n-1)$'th effective element in the sequence.
\end{prop}

One last remarkable property of the $n$-ball projection of the Hilbert-Galton board is that one can compute its full spectrum. Let $M_n$ be the (column-stochastic) generator of $Y^{(n)}$. Then we have the following result.
\begin{thm}
\label{thm:spectrum}
For any $1 \leq j \leq n-1$, $-y_j$ appears as an eigenvalue of $M_n$ exactly $\binom {n-1}{j}$ times. In addition, $0$ is the Perron eigenvalue.
\end{thm}

\begin{rem}
\label{rem:eigenvectors}
In the course of the proof of Theorem~\ref{thm:spectrum}, we will prove that for any $n\geq3$, the matrix $M_n$ is not diagonalizable and we will provide $2^{n-2}$ independent eigenvectors of $M_n$.
\end{rem}

\begin{ex}
\label{ex:eigenvalues}
In the ordered basis $\{(1,1,1),(1,2),(2,1),(3)\}$, the transition matrix for the $Y^{(3)}$ is given by
\[
M_3 = 
\left(
\begin{array}{cccc}
 -x_1-x_2 & x_0 & 0 & 0 \\
 x_2 & -x_0-x_1 & x_0 & x_0 \\
 x_1 & x_1 & -x_0-x_1 & 0 \\
 0 & 0 & x_1 & -x_0 \\
\end{array}
\right),
\]
and the eigenvalues of $M_3$ are $0,-x_0-x_1,-x_0-x_1$ and $-x_0-x_1-x_2$, in agreement with Theorem~\ref{thm:spectrum}.
\end{ex}

\section{The Hilbert-Galton board $Z^{(n)}$ and its enrichment $X^{(n)}$}
\label{sec:binmodel}

In this section we provide two proofs of Theorem~\ref{thm:constantbinprob}. The first one is more probabilistic and uses a coupling of the Hilbert-Galton board with a bi-infinite stationary version of the process. The second one is more combinatorial and relies on the proof of Theorem~\ref{thm:triangle-chain-statioprob} giving the stationary distribution of the enriched Hilbert-Galton board which projects down to the Hilbert-Galton board.

\subsection{Coupling with a bi-infinite stationary version}

We start by proving the irreducibility and positive recurrence of the 
Hilbert-Galton board.

\begin{lem}
\label{lem:Zirreducible}
The Markov chain $Z^{(n)}$ is irreducible and positive recurrent.
\end{lem}
\begin{proof}
Starting from any state $c'\in\Z_{\geq0}^n$, one can get to any given state $c=(c_1,\ldots,c_n)$ by applying the following sequence of transitions of type $U_0$ and $U_1$, which have positive probability since $x_0>0$ and $x_1>0$: apply $U_0$ followed by $c_n$ times $U_1$ then $U_0$ followed by $c_{n-1}$ times $U_1$, etc, until $U_0$ followed by $c_1$ times $U_1$. Thus $Z^{(n)}$ is irreducible.

Let $c_e=(0,\ldots,0)$ be the state with all the bins empty. We will show that the return time to $c_e$ has finite expectation. Assume that $Z^{(n)}(0)=c_e$ and set
\[
\tau_e=\inf\left\{t\geq0 \; | \; Z^{(n)}(t) \neq c_e\right\}.
\]
Then the return time to $c_e$ is given by
\[
R_e=\inf\left\{t\geq\tau_e \; | \; Z^{(n)}(t) = c_e\right\}.
\]
For any integer $k\geq1$, let $E_k$ be the event that between times $k-1$ and $k$, a single transition $U_1$ followed by $n$ transitions $U_0$ have occurred and no other transition has occurred. Then all the events $E_k$ have a common probability denoted by $\eta$. Since $x_0$ and $x_1$ are positive, we have $\eta>0$. Furthermore,
\[
0\leq R_e\leq \min\left\{k\geq1 \; | \; E_k \text{ is realized} \right\}.
\]
Since $R_e$ is dominated by a geometric random variable, it has finite expectation, which implies that $Z^{(n)}$ is positive recurrent.
\end{proof}

We first provide a proof of Theorem~\ref{thm:constantbinprob} based on a coupling of $Z^{(n)}$ with an bi-infinite stationary version $\hat{Z}^{(n)}$.
Consider the following $n+1$ independent Poisson point processes on the real line: for any $0 \leq j \leq n$, let $Q_j$ be a Poisson point process of intensity $x_j$. On the one hand, one can use the restrictions of these Poisson point processes to $\R_{\geq0}$ to define the Markov chain $(Z^{(n)}(t))_{t\geq0}$ : $Z^{(n)}(0)$ is distributed according to some initial probability distribution, and at every time $t>0$ belonging to some $Q_j$ with $0\leq j \leq n$, a transition $U_j$ occurs. The Poisson point process $Q_j$ corresponds to the times at which a ball falls into bin number $j$ (resp. a shift occurs) if $1\leq j\leq n$ (resp. if $j=0$). On the other hand, we use the Poisson point processes on the whole line to define the process $(\hat{Z}^{(n)}(t))_{t\in\R}$ as follows. Fix $t\in\R$, set $\tau_0(t)=t$ and define by induction on $1\leq i\leq n$
\[
\tau_i(t):=\sup \{Q_0 \cap (-\infty,\tau_{i-1}(t) )\}.
\]
The sequence $(\tau_i(t))_{1 \leq i \leq n}$ corresponds the last $n$ shift times before time $t$. Set also for any $1\leq i\leq n$ and 
$1\leq k\leq n-i+1$
\begin{equation}
\label{eq:defB}
B_{i,k}(t)=\#(Q_k \cap (\tau_i(t),\tau_{i-1}(t)))
\end{equation}
The integer $B_{i,k}(t)$ represents the number of balls that fall between the $i$'th shift before time $t$ and the $(i-1)$'th shift before time $t$ in the bin that is at that time at position $k$. We then set
\[
(\hat{Z}^{(n)}(t)):=(\hat{Z}^{(n)}_1(t),\ldots,\hat{Z}^{(n)}_n(t)),
\]
where for any $1\leq i \leq n$,
\begin{equation}
\label{eq:sumB}
\hat{Z}^{(n)}_i(t):=\sum_{j=1}^i B_{i+1-j,j}(t).
\end{equation}
\begin{lem}
\label{lem:biinfinitestationary}
For any $t\in\R$, the random variable $\hat{Z}^{(n)}(t)$ is distributed according to the stationary distribution $\pi_{Z^{(n)}}$ of $Z^{(n)}$.
\end{lem}
\begin{proof}
If $t\geq0$ is such that
\begin{equation}
\label{eq:coupling}
\#(Q_0 \cap (0,t))\geq n
\end{equation}
this means that at least $n$ shifts have occurred between time $0$ and time $t$ thus the state $Z^{(n)}(t)$ can be reconstructed only from the knowledge of the realization of the Poisson point processes between times $0$ and $t$ (not requiring any knowledge of the state at time $0$), exactly in the same way that $\hat{Z}^{(n)}(t)$ was defined. Since the processes $(\hat{Z}^{(n)}(t))_{t\in\R}$ and $(Z^{(n)}(t))_{t\geq0}$ are coupled by the use of the same Poisson point processes for their definition, this means that if $t\geq0$ is such that condition~\eqref{eq:coupling} is satisfied, then $Z^{(n)}(t)=\hat{Z}^{(n)}(t)$ a.s. So for any $t\geq0$,
\begin{equation}
\label{eq:domination}
\mathbb{P}(Z^{(n)}(t) \neq \hat{Z}^{(n)}(t)) \leq \mathbb{P}(\#(Q_0 \cap (0,t))< n).
\end{equation}
Since $x_0>0$, the right-hand side of~\eqref{eq:domination} tends to zero as $t$ goes to infinity. 
As a consequence of Lemma~\ref{lem:Zirreducible}, the distribution of $Z^{(n)}(t)$ converges to its stationary distribution $\pi_{Z^{(n)}}$, thus the distribution of $\hat{Z}^{(n)}(t)$ converges to $\pi_{Z^{(n)}}$ too. We conclude from the fact that the distribution of $\hat{Z}^{(n)}(t)$ is the same for every $t\in\R$ because the distributions of the Poisson point processes used to define $\hat{Z}^{(n)}(t)$ are invariant by translation.
\end{proof}
Now we can use the bi-infinite stationary process $\hat{Z}^{(n)}$ to derive a formula for the distribution $\pi_{Z^{(n)}}$.

\begin{proof}[Proof of Theorem~\ref{thm:constantbinprob}]
In this proof we set $t=0$ and to simplify notation we will write simply $\tau_i$ and $B_{i,k}$ instead of $\tau_i(0)$ and $B_{i,k}(0)$ respectively. Fix $c=(c_1,\ldots,c_n)\in\Z_{\geq0}^n$ and $1\leq \ell \leq n$. By the definition~\eqref{eq:sumB} of $\hat{Z}^{(n)}(0)$, for any $1\leq i \leq n$,
\[
\mathbb{P}(\hat{Z}^{(n)}_i(0)=c_i)=\sum\limits_{\substack{(A_{i,1},\ldots,A_{i,i})\in\mathbb{Z}_{\geq0}^i \\ A_{i,1}+\cdots+A_{i,i}=c_i}} \mathbb{P}(B_{i+1-j,j}=A_{i,j} \text{ for all } 1\leq j \leq i).
\]
More generally, we can express the probability of a distribution of balls in the leftmost $\ell$ bins in terms of triangular arrays as
\begin{multline}
\label{eq:constantbinstatioprob}
\mathbb{P}(\hat{Z}^{(n)}_i(0)=c_i \text{ for all } 1\leq i \leq \ell)= \\
\sum_{A\in H_{\ell}(c_1,\ldots,c_{\ell})}\mathbb{P}(B_{i+1-j,j}=A_{i,j} \text{ for all } (j,i) \text{ such that } 1\leq j \leq i \leq \ell).
\end{multline}
Set $t_i:=\tau_{i-1}-\tau_i$ for all $1 \leq i \leq n$. Conditioned on the value of the $\tau_i$'s, the random variables $(B_{i,k})$ are independent and $B_{i,k}$ is distributed like a Poisson random variable of parameter $t_i x_k$.  We can compute the summand of the right-hand side of~\eqref{eq:constantbinstatioprob} conditionally on $Q_0$, ie conditioned on the $t_i$'s by the formula
\begin{align*}
&\mathbb{P}(B_{i+1-j,j}=A_{i,j} \text{ for all } (j,i) \text{ such that } 1\leq j \leq i \leq \ell \; | \; Q_0) \\
&=\prod_{i=1}^{\ell}\prod_{j=1}^i e^{-t_{i+1-j} x_j} \frac{(t_{i+1-j} x_j)^{A_{i,j}}}{A_{i,j}!}
\end{align*}
Since the $t_i$'s are iid exponential random variables with rate parameter $x_0$, we deduce that
\begin{align*}
&\mathbb{P}(B_{i+1-j,j}=A_{i,j} \text{ for all } (j,i) \text{ such that } 1\leq j \leq i \leq \ell) \\
&=\int_{(t_1,\ldots t_{\ell})\in \mathbb{R}_{\geq0}^{\ell}}\prod_{i=1}^{\ell}
e^{-x_0t_i} \: x_0 \; \text{d}t_i\prod_{j=1}^i e^{-t_{i+1-j} x_j} \frac{(t_{i+1-j} x_j)^{A_{i,j}}}{A_{i,j}!} \\
&= x_0^{\ell} \prod_{j=1}^{\ell} x_j^{v_j(A)} \int_{(t_1,\ldots t_{\ell})\in \mathbb{R}_{\geq0}^{\ell}}\prod_{i=1}^{\ell}e^{-y_{\ell+1-i}t_i} \frac{t_i^{d_{\ell+1-i}(A)}}{\prod_{j=1}^{\ell-i+1} A_{i-1+j,j}!} \text{d}t_i \\
&= x_0^{\ell} \prod_{j=1}^{\ell} x_j^{v_j(A)} \prod_{i=1}^{\ell} \int_{t_i\in \mathbb{R}_{\geq0}}e^{-y_{\ell+1-i}t_i} \frac{t_i^{d_{\ell+1-i}(A)}}{\prod_{j=1}^{\ell-i+1} A_{i-1+j,j}!} \text{d}t_i \\
&= x_0^{\ell} \prod_{j=1}^{\ell} x_j^{v_j(A)} \prod_{i=1}^{\ell} \frac{1}{y_{\ell+1-i}^{1+d_{\ell+1-i}(A)}} \frac{d_{\ell+1-i}(A)!}{\prod_{j=1}^{\ell-i+1} A_{i-1+j,j}!} \\
&=w_x(A).
\end{align*}
Since the distribution of $\hat{Z}^{(n)}(0)$ is $\pi_{Z^{(n)}}$ by Lemma~\ref{lem:biinfinitestationary}, this concludes the proof. 
\end{proof}

\subsection{The enriched Hilbert-Galton board}

In this subsection we will assume that $x_i>0$ for all $0 \leq i \leq n$.
We now provide an alternate proof of Theorem~\ref{thm:constantbinprob} by obtaining it as a corollary of Theorem~\ref{thm:triangle-chain-statioprob} on the stationary distribution of the enriched Hilbert-Galton board $X^{(n)}$.
We begin by showing irreducibility and positive recurrence. 

\begin{lem}
\label{lem:triangle-chain-irred}
The Markov chain $X^{(n)}$ on $\A_n$ is irreducible and positive recurrent.
\end{lem}

\begin{proof}
Starting from any array $A'$, one can reach an arbitrary array $A$ by the following sequence of transitions which all have positive probability: first make a transition of type $0$ followed by $A_{n,1}$ transitions of type $1$, then make a transition of type $0$ followed by $A_{n-1,1}$ transitions of type $1$ and $A_{n,2}$ transitions of type $2$. In this way continue filling up the array diagonally until one arrives at $A$. Therefore, $X^{(n)}$ is irreducible. By the same argument as in the proof of Lemma~\ref{lem:Zirreducible}, if $A_e$ denotes the array filled with zeros, the expected return time to $A_e$ is finite and hence the chain $X^{(n)}$ is positive recurrent.
\end{proof}

Now we prove that the Markov chain $X^{(n)}$ indeed projects down to the Markov chain $Z^{(n)}$.

\begin{lem}
Defining the map
\[
p_n:A\in\A_n\rightarrow (h_1(A),\ldots,h_n(A))\in\Z_{\geq0}^n,
\]
the projection $p_n(X^{(n)})$ is a Markov chain with the same law as $Z^{(n)}$.
\end{lem}
\begin{proof}
Fix $A\in\A_n$. Then for $1\leq i \leq n$, at rate $x_i$ the integer $h_i(A)$ increases by $1$ and at rate $x_0$ the $n$-tuple $(h_1(A),\ldots,h_n(A))$ gets turned into $(0,h_1(A),\ldots,h_{n-1}(A))$. This is exactly the law of $Z^{(n)}$.
\end{proof}

\begin{rem}
The numbers $A_{i,j}$ have the following interpretation in terms of balls and bins. Call an \emph{epoch} the time between two consecutive shifts and count the epochs backwards: the current epoch is indexed by $1$, the epoch just before the current epoch is indexed by $2$, etc. Then $A_{i,i+1-j}$ is the number of balls that are in the bin currently indexed by $i$ which arrived during epoch $j$. In the notation of the previous subsection, we have $A_{i,j}=B_{i+1-j,j}$.
\end{rem}

We will need the following identity to complete the proof of Theorem~\ref{thm:triangle-chain-statioprob}.

\begin{lem}
\label{lem:identity}
Let $\delta_1,\ldots,\delta_{n-1}$ be $n-1$ nonnegative integers. Then
\[
\sum_{(\alpha_1,\dots,\alpha_n) \in \mathbb{Z}_{\geq 0}^n} 
\prod_{k=1}^{n-1} \binom{\delta_k + \alpha_{k+1}}{\alpha_{k+1}} 
\prod_{k=1}^{n} \left(\frac{x_k}{y_k} \right)^{\alpha_k}
= 
\frac{y_n^{\delta_{n-1}+1}}{x_0 y_1^{\delta_1}}
\prod_{k=2}^{n-1} y_k^{\delta_{k-1} - \delta_k}. 
\]
\end{lem}

\begin{proof}
From the generalized binomial theorem, we have that for any real number $\delta$ and for any real number $0\leq z <1$,
\[
\sum_{\alpha=0}^\infty \binom{\delta+\alpha}\alpha z^\alpha = (1-z)^{-\delta-1}.
\]
The sum on the left-hand side splits into $n$ such sums (with $\delta = 0$ when $k=1$).
Recalling that $y_{k-1} + x_k = y_k$ and rearranging the terms gives the right-hand side.
\end{proof}

\begin{lem}
\label{lem:stat prob}
For any $n\geq1$, the measure $w_x:\A_n\rightarrow\mathbb{R}_{\geq0}$ is a probability measure, that is
\[
\sum_{A \in \A_n} w_x(A) = 1.
\]
\end{lem}

\begin{proof}
The proof proceeds by induction. When $n=1$, it is easy to check
that $\sum_{A \in \A_1} w_x(A) = 1$ because it is a geometric series.
Now we want to perform the sum for general $n$. 
Let $A'$ denote the triangular array $A$ restricted to the first $n-1$ rows.
We split the sum into two parts, 
\begin{align*}
& \sum_{A'\in \A_{n-1}}  \sum_{(A_{n,1},\dots,A_{n,n}) \in \mathbb{Z}_{\geq 0}^n} w_x(A) \\
=& \frac{x_0^{n}}{y_1 \cdots y_n}  \sum_{A'\in \A_{n-1}}  \sum_{(A_{n,1},\dots,A_{n,n}) \in \mathbb{Z}_{\geq 0}^n} 
\prod_{k=1}^n \binom{d_k(A)}{A_{n-k+1,1},  \dots , A_{n,k}}
\frac{x_k^{v_k(A)}}{y_k^{d_k(A)}}.
\end{align*}
After rewriting the multinomial
coefficient and shifting the indices, the sum above becomes
\begin{align*}
&\frac{x_0^{n}}{y_1 \cdots y_n}  \sum_{A'\in \A_{n-1}}  
\prod_{k=1}^{n-1} \binom{d_k(A') }{A_{n-k,1},  \dots , A_{n-1,k}} \frac{ x_k^{v_k(A')}}{ y_{k+1}^{d_k(A')}} \\
& \times \sum_{(A_{n,1},\dots,A_{n,n}) \in \mathbb{Z}_{\geq 0}^n} 
\prod_{k=1}^{n-1} \binom{d_{k}(A')+A_{n,k+1}}{A_{n,k+1}}
\prod_{k=1}^{n} \left( \frac{ x_k}{ y_{k}} \right)^{A_{n,k}}.
\end{align*}
We are now in a position to use Lemma~\ref{lem:identity} to evaluate the inner sum, for which 
we obtain
\[
\frac{y_n^{d_{n-1}(A')+1}}{x_0 y_1^{d_1(A')}} \; \prod_{k=2}^{n-1} y_k^{d_{k-1}(A') - d_k(A')}.
\]
After combining these factors with the outer sum, we obtain exactly $\sum_{A'\in \A_{n-1}} w_x(A')$,
which by the induction assumption is 1, completing the proof.
\end{proof}

\begin{proof}[Proof of Theorem~\ref{thm:triangle-chain-statioprob}]
Since $X^{(n)}$ is irreducible and 
positively recurrent by Lemma~\ref{lem:triangle-chain-irred}, the
stationary distribution is unique. Furthermore, $w_x$ is a probability
distribution by Lemma~\ref{lem:stat prob}.
It is therefore sufficient to verify that $w_x(A)$ satisfies the master equation, that is to say, the total
outgoing rate from $A$ is the same as the total incoming rate into $A$. 
Fix $n\geq1$. 
Since $A$ has size $n$, the total outgoing rate is always $y_n w_x(A)$.
There are two cases depending on the value of $d_n(A)$.

{\bf Case I: $d_n(A) = 0:$} 
Therefore $A_{1,1} = \cdots = A_{n,n} = 0$. 
The only incoming transitions into $A$ are with rate $x_0$ from the set
\begin{equation}
\label{TofA}
S(A) = \left\{ 
\begin{array}{ccccc}
 A_{2,1} \\
 A_{3,1} & A_{3,2} \\
 \vdots & \vdots & \ddots  \\
 A_{n,1} & A_{n,2} & \dots &  A_{n,n-1} \\
 \alpha_1 & \alpha_2 & \dots & \dots & \alpha_n
\end{array}
\: : \: (\alpha_1,\dots,\alpha_n) \in \mathbb{Z}_{\geq0}^n \right\}.
\end{equation}
Let $A'\in\A_{n-1}$ be the array obtained from $A$ by deleting its top diagonal:
\[
\begin{array}{cccc}
 A_{2,1} \\
 A_{3,1} & A_{3,2} \\
 \vdots & \vdots & \ddots  \\
 A_{n,1} & A_{n,2} & \dots &  A_{n,n-1}
\end{array}
\]
Then by the same computation as in the proof of Lemma~\ref{lem:stat prob}, we have $\sum_{\tilde{A} \in S(A)} w_x(\tilde{A})=w_x(A')$. Since the top diagonal of $A$ contains only $0$'s, we also have $w_x(A)=w_x(A')\times x_0 / y_n$.
So finally we have
\[
x_0 \sum_{\tilde{A} \in S(A)} w_x(\tilde{A}) = y_n w_x(A),
\]
which is exactly the master equation for $A$.

{\bf Case II: $d_n(A) > 0:$} 
Let $I(A) = \{i \in [1,n]: A_{i,i} > 0 \}$.  
Let $A^{(i)}$ be the same array as $A$, with the sole difference
being that $A_{i,i}$ is replaced by $A_{i,i}-1$. Note that $A^{(i)}$ makes a transition to $A$ with rate $x_i$ for each $i \in I(A)$.
Using the definition~\eqref{eq:wx} of $w_x(A)$, it is clear that
\[
x_i  w_x(A^{(i)}) = y_n w_x(A) \frac {A_{i,i}}{d_n(A)}.
\] 
Summing over all $i \in I$ and using the fact that $\sum_{i \in I(A)} A_{i,i} = d_n(A)$ gives the master equation at $A$.
\end{proof}

Theorem~\ref{thm:constantbinprob} follows as a corollary of Theorem~\ref{thm:triangle-chain-statioprob} in the case when $x_i>0$ for all $i\geq0$. Since the master equation for the chain $Z^{(n)}$ is rational in the $x_i$'s, one can actually pass to the limit when $x_i$ goes to $0$ for all $i\geq2$ and recover Theorem~\ref{thm:constantbinprob} in full generality.

\section{The $n$-ball projection of the Hilbert-Galton board}
\label{sec:finiteballprojection}

\subsection{Stationary distribution}

\begin{lem}
\label{lem:irreducibility}
The $n$-ball projection of the Hilbert-Galton board, $Y^{(n)}$, is irreducible.
\end{lem}

\begin{proof}
Starting from any composition $c'\in C_n$, one gets to the composition $c=(c_1,\ldots,c_\ell)\in C_n$ by the following sequence of transitions, each having positive probability: $T_0$ followed by $c_\ell-1$ times $T_1$, then $T_0$ followed by $c_{\ell-1}-1$ times $T_1$, etc, until $T_0$ followed by $c_1-1$ times $T_1$.
\end{proof}

One can couple $Y^{(n)}$ with $Z^{(n)}$ as follows. Define the maps
\begin{align*}
f_n:&(c_1,\ldots,c_n)\in\Z_{\geq0}^n\mapsto(c_1+1,\ldots,c_n+1)\in\Z_{>0}^n \\
\ell_n:&(c_1,\ldots,c_n)\in\Z_{>0}^n\mapsto
\min\left\{1\leq i \leq n \; | \; c_1+\cdots+c_i\geq n\right\} \\
g_n:&(c_1,\ldots,c_n)\in\Z_{>0}^n\mapsto (c_1,\ldots,c_{\ell_n(c)-1},n-(c_1+\cdots+c_{\ell_n(c)-1})),
\end{align*}
where in the last definition $c$ was shorthand for $(c_1,\ldots,c_n)$. Let $\tilde{Z}^{(n)}$ be the pushforward of $Z^{(n)}$ by $f_n$ and $\tilde{Y}^{(n)}$ be the pushforward of $\tilde{Z}^{(n)}$ by $g_n$. Then $\tilde{Y}^{(n)}$ has the same law as $Y^{(n)}$, so that one can use Theorem~\ref{thm:constantbinprob} to study $Y^{(n)}$. In particular, the stationary distribution $\pi_{Y^{(n)}}$ is obtained as a pushforward by $g_n\circ f_n$ of the stationary distribution $\pi_{Z^{(n)}}$.

\begin{proof}[Proof of Theorem~\ref{thm:statioprob}]
Fix $c=(\gamma_1+1,\ldots,\gamma_{\ell}+1)\in C_n$ to be a composition of $n$ of length $\ell$. Then 
\begin{align}
\pi_{Y^{(n)}}(c)&=\pi_{Y^{(n)}}(Y^{(n)}_1=\gamma_1+1,\ldots,Y^{(n)}_{\ell}=\gamma_{\ell}+1) \\
&=\pi_{Z^{(n)}}(Z^{(n)}_1=\gamma_1,\ldots,Z^{(n)}_{\ell-1}=\gamma_{\ell-1},Z^{(n)}_{\ell}\geq \gamma_{\ell}).
\end{align}
Writing
\begin{multline*}
\pi_{Y^{(n)}}(c)=\pi_{Z^{(n)}}(Z^{(n)}_1=\gamma_1,\ldots,Z^{(n)}_{\ell-1}=\gamma_{\ell-1}) \\
-\sum_{s=0}^{\gamma_{\ell}-1} \pi_{Z^{(n)}}(Z^{(n)}_1=\gamma_1,\ldots,Z^{(n)}_{\ell-1}=\gamma_{\ell-1},Z^{(n)}_{\ell}=s)
\end{multline*}
and applying formula~\eqref{eq:constantbinprob}, we obtain \eqref{eq:statioprob1}. Writing
\[
\pi_{Y^{(n)}}(c)=\sum_{s=\gamma_l}^{\infty}  \pi_{Z^{(n)}}(Z^{(n)}_1=\gamma_1,\ldots,Z^{(n)}_{\ell-1}=\gamma_{\ell-1},Z^{(n)}_{\ell}=s)
\]
and applying formula~\eqref{eq:constantbinprob}, we obtain \eqref{eq:statioprob2}.
\end{proof}

We now prove the formula for the partition function.

\begin{proof}[Proof of Corollary~\ref{cor:partitionfunction}]
Fix a composition $c=(\gamma_1+1,\ldots,\gamma_{\ell}+1)\in C_n$ and consider the formula for $\pi_{Y^{(n)}}(c)$ given by formula~\eqref{eq:statioprob1}. For any $1 \leq k \leq \ell \leq n$, denote by $\mu_k(c)$ the maximal power with which $y_k$ appears in the denominator of any of the monomials in the formula for $\pi_{Y^{(n)}}(c)$. We will show that $\mu_k(c)\leq n-k$. If $k=\ell=n$, then this means that the composition $c$ is equal to $c_n(1)$ and by Corollary~\ref{cor:monomialproba}, we have $\mu_n(c)=0$. Assume next that $k=\ell < n$. Then $y_{\ell}$ does not appear in a monomial in the first sum, thus
\begin{align*}
\mu_{\ell}(c)&\leq 1+\max_{0 \leq s \leq \gamma_{\ell}-1} \max_{A\in H_{\ell}(\gamma_1,\ldots,\gamma_{\ell-1},s)} d_k(A) \\
& \leq 1+\gamma_{\ell}-1+ \sum_{j=1}^{\ell-1} \gamma_{\ell-j} \\
& \leq n - \ell.
\end{align*}
Finally, if $1 \leq k \leq \ell-1$, then
\begin{align*}
\mu_k(c)&\leq 1+\max\left(\max_{A\in H_{\ell-1}(\gamma_1,\ldots,\gamma_{\ell-1})} d_k(A),\max_{0 \leq s \leq \gamma_{\ell}-1} \max_{A\in H_{\ell}(\gamma_1,\ldots,\gamma_{\ell-1},s)} d_k(A)\right) \\
& \leq 1+\max\left(\sum_{j=1}^k \gamma_{\ell-j},\gamma_{\ell}-1+ \sum_{j=1}^{k-1} \gamma_{\ell-j}\right) \\
& \leq 1 + n - \ell \\
& \leq n-k.
\end{align*}
On the other hand, by Corollary~\ref{cor:monomialproba}, for any $1 \leq k \leq n-1$, $c_{k+1}(n-k)$ is a composition of $n$ whose stationary probability consists of a single monomial, and the power of $y_k$ in the denominator is equal to $n-k$. We conclude that the power of $y_k$ in the least common multiple of the denominators of the stationary probabilities of compositions of length $n$ is equal to $n-k$.
\end{proof}

\begin{rem}
By analyzing the expression for $\mu_k(c)$, , 
one can show that if $c=(\gamma_1+1,\ldots,\gamma_{\ell}+1)\in C_n$:
\begin{itemize}
 \item for any $1 \leq k \leq \ell-1$,
 \[
 \mu_k(c)= \max\left(\gamma_{\ell-k}+1,\gamma_{\ell}\right) + \sum_{j=1}^{k-1} \gamma_{\ell-j} 
 \]
 \item when $k=\ell$,
 \[
 \mu_{\ell}(c)=
 \begin{cases}
  0 &\text{ if } \gamma_{\ell}=0 \\
  n- \ell &\text{ if } \gamma_{\ell}>0
 \end{cases}.
 \]
\end{itemize}
\end{rem}

\subsection{Grand coupling time}
\label{subsec:coupling}
In this subsection we prove Proposition~\ref{prop:couplingtime} giving an expression for the grand coupling time associated with a given sequence of transitions.

\begin{proof}[Proof of Proposition~\ref{prop:couplingtime}]
We say a bin is {\em recent} (resp. {\em ancient}) if it is the leftmost bin of the configuration at time $0$ or if it is created after time $0$ (resp. if it already existed at time $0$, but was not the leftmost bin at that time). It then follows that effective (resp. ineffective) elements of the sequence are those that create a ball in a recent (resp. ancient) bin. By looking at the two initial configurations $(n)$ and $(1,n-1)$, it is easy to see that one needs exactly $n-1$ transitions to couple them. 
Moreover, for any initial configuration, the leftmost bin contains at least one ball. Therefore, once one knows how $n-1$ balls have been added in recent bins, one knows the entire configuration, regardless of what the initial configuration was. 
So $n-1$ effective elements are enough to couple any pair of initial configurations.
\end{proof}

\subsection{Spectrum}
\label{subsec:spectrum}

The main idea of the proof of Theorem~\ref{thm:spectrum} is that $M_n$, when written in a different basis, becomes block upper-triangular, with one of its diagonal blocks being a diagonal matrix and the other one being $M_{n-1}$. In order to observe this, we first need to define an order on the elements of $C_n$ making it compatible with that on $C_{n-1}$.

We define the maps $L$ and $R$ from $C_n$ to $C_{n+1}$ for any $n \geq 1$ as follows. For any $c=(c_1,\ldots,c_{\ell})\in C_n$,
\begin{align*}
L(c) &:= (c_1,\ldots,c_{\ell},1) \\
R(c) &:= (c_1,\ldots,c_{\ell}+1)
\end{align*}
The maps $L$ and $R$ both add a single ball to the right of the configuration, either in a new rightmost bin or in the current rightmost bin. We also define the map $D:C_n\rightarrow C_{n-1}$ which erases the rightmost ball: for any $c=(c_1,\ldots,c_{\ell})\in C_n$,
\[
D(c):=
\begin{cases}
(c_1,\ldots,c_{\ell}-1) &\text{ if } c_{\ell}\geq 2 \\
(c_1,\ldots,c_{\ell-1}) &\text{ if } c_{\ell}=1
\end{cases}.
\]
We proceed by induction on $n \geq 1$ to assign a rank to each composition in $C_n$. We set $s_1^{(1)}:=(1)$ and for any $n \geq 2$ and $1\leq i \leq 2^{n-2}$, we set
\begin{align}
s_{2i-1}^{(n)} &:= L(s_i^{(n-1)}) \\
s_{2i}^{(n)} &:= R(s_i^{(n-1)}) 
\end{align}
These compatibility relations can be illustrated by the rooted plane binary tree on Figure~\ref{fig:tree}, where the order on $C_n$ can be read from left to right on the $n$'th row.
\begin{figure}[htbp]
\centering
\includegraphics[height=2in]{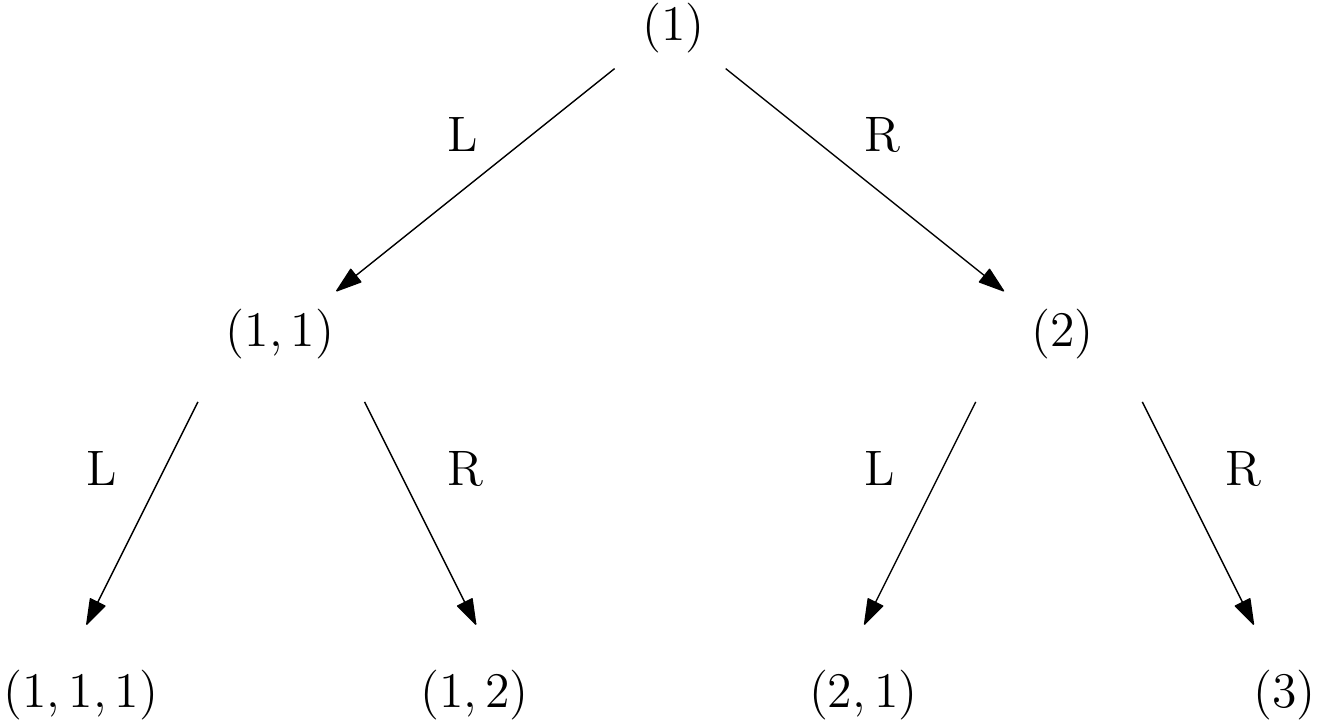}
\caption{Compatibility between the orders.}
\label{fig:tree}
\end{figure}
By construction, we have for any $n\geq 2$ and $1\leq i \leq 2^{n-2}$,
\begin{equation}
D(s_{2i-1}^{(n)}) = D(s_{2i}^{(n)})= s_i^{(n-1)}
\end{equation}
\begin{rem}
It is not hard to show that this order on $C_n$ defined inductively is actually the lexicographic order on $C_n$, where a composition $c=(c_1,\ldots,c_\ell)\in C_n$ can be seen as a word of $\ell$ letters on the alphabet $\mathbb{N}$ equipped with its usual order.
\end{rem}
We now define two bases on $\mathbb{R}^{C_n}$. The first one is the basis $\calB_n=(e(s_i^{(n)}))_{1 \leq i \leq 2^{n-1}}$, where each basis vector is indexed by a composition. The second basis is $\calB'_n=(f_1^{(n)} ,\ldots,f_{2^{n-2}}^{(n)} ,g_1^{(n)} ,\ldots,g_{2^{n-2}}^{(n)} )$ where for any $1 \leq i \leq 2^{n-2}$, we define
\begin{align}
f_i^{(n)}  &:= e(s_{2i}^{(n)}) - e(s_{2i-1}^{(n)}) \\
g_i^{(n)}  &:= e(s_{2i-1}^{(n)}) + e(s_{2i}^{(n)})
\end{align}
Let $u_n$ be the endomorphism of $\mathbb{R}^{C_n}$ whose matrix in the basis $\calB_n$ is given by $M_n$. Denote by $M'_n$ the matrix of $u_n$ in the basis $\calB'_n$. Then we can find the eigenvalues of $M_n$ by induction on $n$, using the following result.

\begin{lem}
\label{lem:blockmatrix}
For any $n \geq 2$, write
\[
M'_n=
\left(
\begin{array}{cccc}
A_{n,1} & A_{n,2} \\
A_{n,3} & A_{n,4} \\
\end{array}
\right),
\]
where the blocks $A_{n,i}$ are all square matrices of size $2^{n-2}$. Then the following holds:
\begin{enumerate}
 \item $A_{n,1}$ is a diagonal matrix, where for every $1 \leq j \leq n-1$, the entry $-y_j$ appears exactly $\binom {n-2}{j-1}$ times as a diagonal element ; \label{block1}
 \item $A_{n,3}=0$ ; \label{block3}
 \item $A_{n,4}=M_{n-1}$. \label{block4}
\end{enumerate}
\end{lem}

\begin{proof}[Proof of Theorem~\ref{thm:spectrum}]
Assuming Lemma~\ref{lem:blockmatrix}, we have that $M'_n$ is a block upper triangular matrix. By induction, we know that eigenvalues of $M_{n-1}$ are $-y_j$ with multiplicity $\binom {n-2} {j}$ for any $1 \leq j \leq n-2$, and $0$ with multiplicity $1$. Thus, we deduce that the eigenvalues of $M'_n$ are $-y_j$ with multiplicity $\binom {n-1} {j}$ for any $1 \leq j \leq n-1$, and $0$ with multiplicity $1$. It also immediately follows from Lemma~\ref{lem:blockmatrix} that for any $n\geq3$, the matrix $M_n$ is not diagonalizable : we easily check that the matrix $M_3$ (see Example~\ref{ex:eigenvalues}) is not diagonalizable and we conclude by induction, since $M_n$ appears as a block of the block triangular matrix $M'_{n+1}$.
\end{proof}

To complete the proof of the theorem, it remains to prove Lemma~\ref{lem:blockmatrix}.

\begin{proof}[Proof of Lemma~\ref{lem:blockmatrix}]
Recall that the map $T_j:C_n \rightarrow C_n$ is given for any $0 \leq j \leq n-1$ by 
\[
T_j(c)=
\begin{cases}
(1,c_1,\dots,c_{\ell-1},c_{\ell}-1) & \text{if $j=0$}, \\
(c_1,\dots,c_{j-1},c_j+1,c_{j+1},\dots,c_{\ell-1},c_{\ell}-1) & \text{if $1 \leq j \leq \ell -1$}, \\
c & \text{if $j \geq \ell$}.
\end{cases}
\]
The action of endomorphism $u_n$ on the basis $\calB_n$ is given by
\begin{equation}
\label{eq:endo}
u_n(e(s_i^{(n)}))=\sum_{j=0}^{n-1} x_j \left(e(T_j(s_i^{(n)})) - e(s_i^{(n)}) \right)
\end{equation}
for any $1 \leq i \leq 2^{n-1}$.
Fix $1 \leq i \leq 2^{n-2}$. Denote by $\ell$ the length of the composition $s_{2i}^{(n)}$. Then by construction, the composition $s_{2i-1}^{(n)}$ has length $\ell+1$. Pick $1\leq j \leq n-1$. The following equalities hold.
\begin{align}
T_j(s_{2i-1}^{(n)}) &=T_j(s_{2i}^{(n)}) && \text{if $j \leq \ell$},\label{eq:merge} \\
T_j(s_{2i-1}^{(n)}) &= s_{2i-1}^{(n)} && \text{if $j \geq \ell+1$}, \label{eq:oddfixedpoint} \\
T_j(s_{2i}^{(n)}) &= s_{2i}^{(n)} && \text{if $j \geq \ell$}. \label{eq:evenfixedpoint}
\end{align}
Combining equations~\eqref{eq:endo}, \eqref{eq:merge}, \eqref{eq:oddfixedpoint} and  \eqref{eq:evenfixedpoint}, we deduce that
\begin{equation}
u_n(f_i^{(n)} )=-y_{\ell} f_i^{(n)} .
\end{equation}
Thus for any $1\leq i \leq 2^{n-2}$, $f_i^{(n)} $ is an eigenvector of $u_n$, and the corresponding eigenvalue is $-y_{\ell}$, where $\ell$ is the length of the composition $s_{2i}^{(n)}$. Using the map $R$ as a length-preserving bijection from $C_{n-1}$ to the subset $\left\{s_{2i}^{(n)} \right\}_{1 \leq i \leq 2^{n-2}}$ of compositions of $n$ of even rank, we deduce the number of compositions of $n$ of even rank that have length $\ell$ equals the number of compositions of $n-1$ that have length $\ell$, which is well-known to be $\binom {n-2}{\ell -1}$. Statements~\ref{block1} and~\ref{block3} of the lemma follow.

Combining equations~\eqref{eq:endo}, \eqref{eq:merge}, \eqref{eq:oddfixedpoint} and  \eqref{eq:evenfixedpoint} again, we obtain that
\begin{equation}
u_n(g_i^{(n)})=\sum_{j=0}^{\ell-1} x_j\left(2e(T_j(s_{2i}^{(n)})) - g_i^{(n)} \right) + x_{\ell}f_i^{(n)}
\end{equation}
It is not hard to see that for any $0 \leq j \leq \ell -1$,
\begin{equation}
\label{eq:commutation}
D(T_j(s_{2i}^{(n)}))=T_j(D(s_{2i}^{(n)}))=T_j(s_{i}^{(n-1)})
\end{equation}
Thus for any $0 \leq j \leq \ell -1$,
\[
T_j(s_{2i}^{(n)}) \in \left\{L(T_j(s_{i}^{(n-1)})), R(T_j(s_{i}^{(n-1)})) \right\}.
\]
If for any $0 \leq j \leq \ell -1$, we call $\tau_j(i)$ the integer such that
\[
T_j(s_{i}^{(n-1)})=s_{\tau_j(i)}^{(n-1)},
\]
then
\[
2e(T_j(s_{2i}^{(n)}))=g_{\tau_j(i)}^{(n)} \pm f_{\tau_j(i)}^{(n)}.
\]
Thus
\[
u_n(g_i^{(n)})=\sum_{j=0}^{\ell -1} x_j \left(g_{\tau_j(i)}^{(n)} - g_i^{(n)}\right) + \sum_{j=0}^{\ell -1} \pm x_j f_{\tau_j(i)}^{(n)} + x_{\ell} f_i^{(n)}
\]
Observing that the matrix $A_{n,4}$ is obtained by considering the projection of $u_n(g_i^{(n)})$ on the subspace spanned by $\left\{g_k^{(n)}\right\}_{1 \leq k \leq 2^{n-2}}$ and comparing with the fact that
\begin{align*}
u_{n-1}(e(s_i^{(n-1)})) &= \sum_{j=0}^{\ell -1} x_j \left(e(T_j(s_i^{(n-1)}))- e(s_i^{(n-1)}) \right) \\
&= \sum_{j=0}^{\ell -1} x_j \left(e(s_{\tau_j(i)}^{(n-1)}))- e(s_i^{(n-1)}) \right)
\end{align*}
we conclude that statement~\ref{block4} holds true.
\end{proof}

As an immediate consequence of Theorem~\ref{thm:spectrum}, we obtain the relaxation time
and the spectral gap.

\begin{cor}
\label{cor:spectral gap}
The spectral gap of the chain is given by $(x_0+x_1)/y_n$ and the relaxation time is $y_n/(x_0+x_1)$.
\end{cor}

\section*{Acknowledgements}
We thank Bastien Mallein for fruitful discussions. 
The first author was partially supported by the UGC Centre for Advanced Studies. The first author also acknowledges support from the Fondation Sciences Math\'e\-matiques de Paris for a visit to the Institut Henri Poincar\'e and from DST
grant DST/INT/SWD/VR/P-01/2014. The second author acknowledges the support and hospitality of the Institut Henri Poincar\'e, where this work was initiated during the program on ``Combinatorics and interactions'', as well as the support of the Fondation Simone et Cino Del Duca.

\bibliographystyle{alpha}
\bibliography{compos}

\end{document}